\documentclass{amsart}

\usepackage[english]{babel}

\usepackage{a4wide}
\usepackage{amsmath}
\usepackage{graphicx}
\usepackage[colorlinks=true, allcolors=blue]{hyperref}

\def\A{\mathcal{A}}

\def\dt{\textrm{d}t}
\def\d{\textrm{d}}

\newcommand{\norm}[1]{\left\lvert#1\right\rvert}

\newcommand{\R}{\mathbb{R}}                  
\newcommand{\Z}{\mathbb{Z}}

\theoremstyle{plain}
\newtheorem{theorem}{Theorem}[section]

\newtheorem{proposition}{Proposition}[section]
\newtheorem{corollary}{Corollary}[section]

\theoremstyle{definition}

\theoremstyle{remark}
\newtheorem{remark}{Remark}[section]

\numberwithin{equation}{section}

\title{A generalization of Gordon's Theorem to arbitrary winding number}
\author{Marco Fenucci}

\address{Marco Fenucci \newline
Planetary Defence Office / NEO Coordination Centre, ESA ESRIN, Largo Galileo Galilei 1, Frascati, 00044, RM, Italy}
\email{marco.fenucci@ext.esa.int}

\begin{document}

\begin{abstract}
We study minimizers of the Keplerian action functional with homogeneous potential of type $1/|u|^\alpha$ on periodic loops with fixed winding number $k \in \Z\setminus\{0\}$. This extends the classical theorems of Gordon and Venturelli, closing the previously open case $\alpha \in (1,2)$, $|k|\geq 2$. We show that the $k$-circular solution loses local minimality below the threshold $\alpha = 2-1/|k|^2$, and that minimizers are collision-free above $\alpha = 2-1/|k|$. Combining these results with the existence of precessing periodic orbits in between, we obtain a full classification of the minimizer in terms of $\alpha$ and $k$, which we also verify numerically for $k=2$.
\textbf{MSC 2020 classification codes:} 70F16 -- 70G75 -- 34C25 
\end{abstract}
\maketitle

\section{Introduction}
Fixed $T>0, \alpha >0$, we define the Lagrangian action of the Kepler problem with $\alpha$-homogeneous potential as 
\begin{equation}
    \A^{\alpha}(u) = \int_0^T \bigg( \frac{1}{2}\norm{\dot{u}}^2 + \frac{1}{\norm{u}^\alpha}\bigg) \dt,
    \label{eq:action}
\end{equation}
where $u:[0,T] \to \R^2$ is a periodic loop of period $T$ in $H^1(\R, \R^2 \setminus\{0\}])$. Given an integer $k \in \Z \setminus\{0\}$, we define the open set
\begin{equation}
    \Omega_k = \{ H^1(\R, \R^2 \setminus \{0\}): u(0) = u(T), \ \deg(u, 0)=k \}.
    \label{eq:Omega_k}
\end{equation}
The problem of Gordon \cite{gordon_1977} consists in finding the minimizers of the action functional $\A^\alpha$ on $\Omega_k$. It is known that the functional of Eq.~\eqref{eq:action} is lower semi-continuous and coercive on $\Omega_k$, therefore by Tonelli's Theorem there exists a minimum on the closure $\overline{\Omega}_k$.

In his paper, Gordon defined two regimes determined by the value of $\alpha$: the weak force case ($0<\alpha<2$), and the strong force case ($\alpha \geq 2$). In the strong force case, collision solutions produce an infinite action, therefore the minimum exists in $\Omega_k$ for all $k$, and the minimizers are smooth and collision free. On the other hand, collision solutions in the weak force case produce a finite action, and therefore they compete to be minimizers. Note that the case classical Keplerian case $\alpha=1$ falls within the weak force case. Gordon proved the following theorem \cite{gordon_1977}.
\begin{theorem}[Gordon]
    If $k=\pm 1$, then the minimizers of $\A^1$ are all the Keplerian solutions, including the collision-ejection solution. If $k\neq \pm 1$, the minimizers of $\A^1$ are only the collision-ejection solutions.
\end{theorem}
\noindent This result was further generalized by Venturelli \cite{venturelli-2002}, who proved the following theorem.
\begin{theorem}[Venturelli]
    For $\alpha \in (0,1)$, the minimizers of $\A^{\alpha}$ on $\overline{\Omega}_k$ are the collision-ejection solutions for every $k \in \Z \setminus \{0\}$. For $\alpha \in (1,2)$ and $k = \pm 1$, the minimizers of $\A^{\alpha}$ are the circular solutions. 
\end{theorem}

\noindent The problem of finding the minimizer of $\A^{\alpha}$ for $\alpha \in (1,2)$ for $k \neq \pm 1$ is still open, and we address it here. 
The difficulty of this remaining range is of a different nature than the cases already solved by Gordon and Venturelli. For $k = \pm 1$ the circular solution is always free of conjugate points, and the only competitor is the collision-ejection solution, so a direct comparison of actions is enough. For $k \geq 2$, however, the circular solution with winding number $k$ can itself become degenerate as a critical point, and additional families of periodic solutions, winding $k$ times around the origin while oscillating radially, enter into competition. A full classification therefore requires controlling three qualitatively different types of orbits simultaneously: the circular solution, the collision-ejection solution, and these intermediate precessing solutions.

Our approach combines three ingredients. First, we compute explicitly the Jacobi fields around the circular solution and show that a conjugate point appears in the interior of the period interval exactly when $\alpha < 2 - 1/|k|^2$. Second, we adapt the blow-up technique to show that minimizers are necessarily free of collisions whenever $\alpha > 2 - 1/|k|$, a strictly weaker threshold. Third, we prove the existence of periodic precessing solutions winding $k$ times around the origin exactly in the range between these two thresholds, using a continuity argument. Combining these three results yields a complete classification of the minimizer of $\A^\alpha$ on $\Omega_k$ for every $\alpha \in (1,2)$ and every $k \in \Z \setminus \{0\}$: the minimizer is the collision-ejection solution, a precessing periodic solution, or the $k$-circular solution, according to where $\alpha$ falls relative to the two thresholds $2 - 1/|k|$ and $2 - 1/|k|^2$ (Theorem~\ref{thm:gordon_general}). We complement this theoretical classification with a numerical study for $k = 2$, in which we compute the action of all three competing families and confirm that the transition between minimizing regimes occurs exactly where the theory predicts.

\section{The $k$-circular solution}
The circular solution is always a solution to the Euler-Lagrange equation associated to the action of Eq.~\eqref{eq:action}, and it always belongs to $\Omega_k$ by just repeating the circle $k$ times. We refer to this solution as $k$-circular solution. In \cite{fenucci_2022}, numerical computations showed that the $k$-circular solution presents conjugate points in $(0,T)$ for certain values of $\alpha$. Here we give a formal proof. 

\begin{proposition}
    For $\alpha \in (1,2)$ and $k \neq \pm 1$, the circular solution has a conjugate point in $(0,T)$ if and only if $\alpha \leq 2 - 1/|k|^2$. 
\end{proposition}
\begin{proof}
    The Lagrangian of the functional of Eq.~\eqref{eq:action} in polar coordinates $(r, \theta)$ is 
    \begin{equation}
        L(r, \dot{r}, \dot{\theta}) = \frac{1}{2}\dot{r}^2 + \frac{1}{2}r^2\dot{\theta}^2 + \frac{1}{r^\alpha}.
    \end{equation}
    The $k$-circular orbit is obtained for $r=r_k$ and $\theta = \omega t$, where $r_k\omega^2 = \alpha r_k^{-\alpha-1}$. The Lagrangian of the second variation around the $k$-circular solution is
    \begin{equation}
        \mathcal{L}(\rho, \dot{\rho}, \dot{\varphi}) = \frac{1}{2}\dot{\rho}^2 + \frac{1}{2}r_k^2 \dot{\varphi}^2 + 2 r_k \omega \rho\dot{\varphi} + \frac{1}{2}(\alpha + 2) \omega^2\rho^2.
    \end{equation}
    The variable $\varphi$ is cyclic, therefore $\partial\mathcal{L}/\partial\dot{\varphi} = r_k^2\dot{\varphi} + 2r_k \omega\rho = C$, where $C$ is a constant. The Euler-Lagrange equation of $\mathcal{L}$ therefore reduces to a forced harmonic oscillator
    \begin{equation}
        \ddot{\rho} + \nu^2\rho = \frac{2\omega C}{r_k}, \quad \nu^2 = (2-\alpha)\omega^2.
    \end{equation}

    For a Jacobi field, we need initial conditions $(\rho(0), \varphi(0)) = (0,0)$ and $(\dot{\rho}(0), \dot{\varphi}(0)) = (A,B)$, the solution integrates to
    \begin{equation}
        \begin{cases}
            \displaystyle \rho(t)  = \frac{A}{\nu} \sin(\nu t) + \frac{2\omega r_k B}{\nu^2}\big( 1- \cos(\nu t) \big), \\[2ex]
            \displaystyle \varphi(t)  = \frac{2 \omega A }{r_k \nu^2} \big( \cos(\nu t) -1 \big) + B \bigg[ \big( 1-\frac{4\omega^2}{\nu^2} \big)t + \frac{4 \omega^2}{\nu^3} \sin(\nu t) \bigg]. 
        \end{cases}
    \end{equation}
    The fundamental solution $Y_0(t)$ of the Jacobi equation, obtained by the two solutions with initial conditions $(A,B)=(1,0)$ and $(A,B) =(0,1)$, is then
    \begin{equation}
        Y_0(t) = \begin{pmatrix} \dfrac{1}{\nu}\sin(\nu t) & \dfrac{2\omega r_k}{\nu^2}\big[1-\cos(\nu t)\big] \\[2mm] \dfrac{2\omega}{r_k\nu^2}\big[\cos(\nu t)-1\big] & \Big(1-\dfrac{4\omega^2}{\nu^2}\Big)t+\dfrac{4\omega^2}{\nu^3}\sin(\nu t) \end{pmatrix}.
    \end{equation}
    With few computations ans substituting $\tau = \nu t$ we obtain:
    \begin{equation}
    \det Y_0(t) = \frac{1}{(2-\alpha)\,\nu^2}\;g(\tau),\qquad g(\tau) = 8(1-\cos\tau)-(\alpha+2)\,\tau\sin\tau.
    \end{equation}
    \begin{figure}
        \centering
        \includegraphics[width=0.8\linewidth]{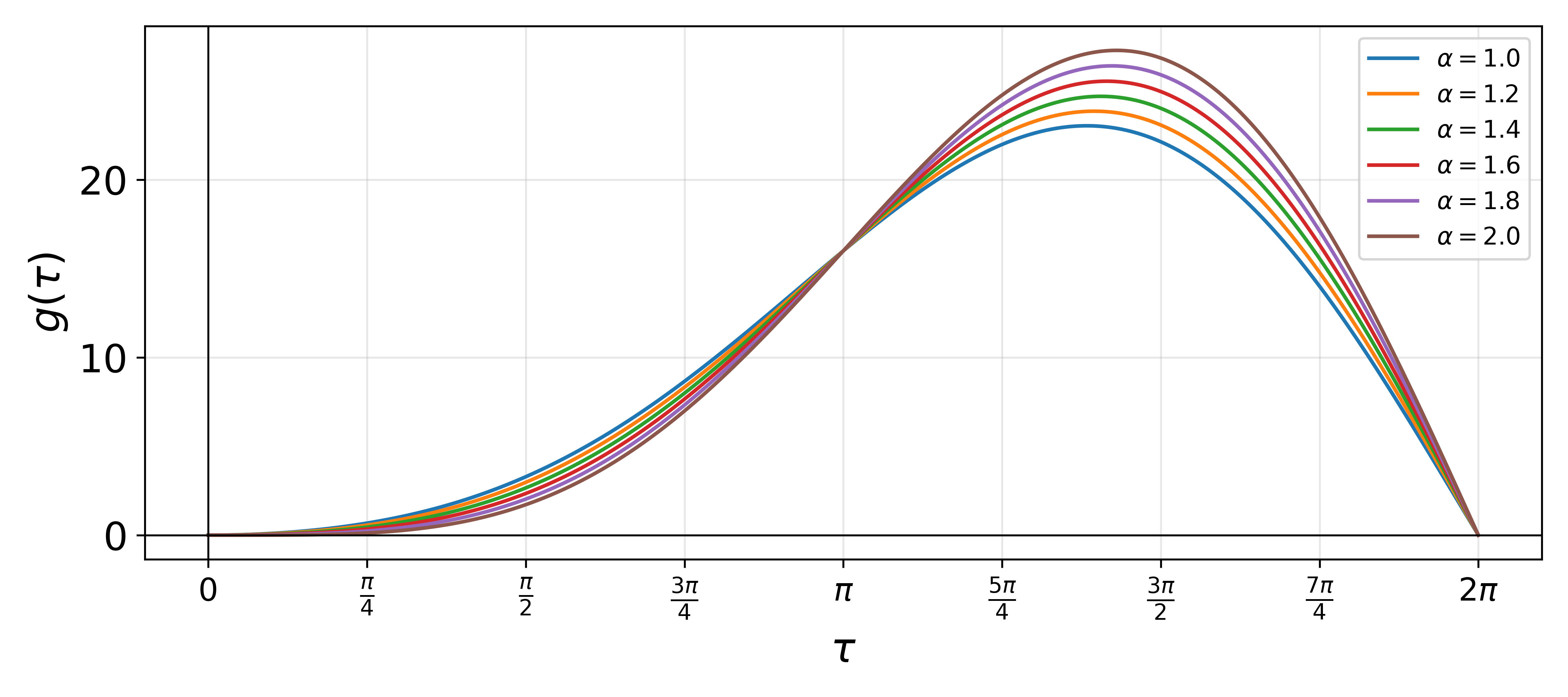}
        \caption{The function $g(\tau)$ for different values of $\alpha \in (1,2)$.}
        \label{fig:gfun}
    \end{figure}
    The function $g(\tau)$ is shown in Fig.~\ref{fig:gfun} for some value of $\alpha$ for a graphic representation. Near $\tau=0$, $g(\tau) = (2-\alpha)\tau^2 >0$, and $g(2\pi)=0$. For $\tau  \in [\pi, 2\pi)$, $g(\tau)>0$ because it is a sum of two positive terms. 
    For $\tau \in (0,\pi]$, note that 
    $g(\tau)>8(1-\cos\tau)-4\tau\sin\tau>0$,
    therefore $g(\tau)>0$ on the whole $[0,2\pi)$. 
    Then the first conjugate point is always at $\tau^* = 2\pi$, which corresponds to $t^* = 2\pi/(\omega\sqrt{2-\alpha})$. Therefore the $k$-circular orbit has no conjugate points if and only if $\omega T < 2\pi/\sqrt{2-\alpha}$. Using $\omega T = 2k\pi$ we get 
    \begin{equation}
        \alpha > 2-\frac{1}{k^2}.
    \end{equation}
\end{proof}

\begin{corollary}
     For $\alpha \in (1, 2 - 1/|k|^2)$, the $k$-circular solution is not a minimizer of $\A^{\alpha}$.
     \label{cor:circ_not_min}
\end{corollary}
\begin{proof}
    By the theory of local minimizers, a necessary condition for a smooth solution of the Euler-Lagrange equation to be a directional local minimum point is for it to be free of conjugate points \cite{fenucci_2022}. Therefore, the $k$-circular solution for $\alpha \in (1, 2 - 1/k^2)$ is not a directional local minimum point, but rather a saddle point. 
\end{proof}

\section{$q$-precessing periodic motions}
For $\alpha \in (1,2)$, bounded motions exist but are generally not closed. The precessing angle $\Phi$ swept by passing from the apocenter to the next pericenter is between $\pi/\sqrt{2-\alpha}$ and $\pi/(2-\alpha)$ (e.g. \cite{castelli-2014}). Therefore, precessing closed motions with exactly one radial oscillation per period, which revolve $k$ times around $0$, could exist only for
\begin{equation}
    2-\frac{1}{|k|} < \alpha < 2-\frac{1}{|k|^2}.
\end{equation}
Precessing motions with $q$ oscillations per period which revolve $k$ times around $0$, could exist only for
\begin{equation}
    2-\frac{q}{|k|} < \alpha < 2-\frac{q^2}{|k|^2}.
    \label{eq:alpha_q}
\end{equation}
We refer to these solutions as $q$-precessing motions.

\begin{proposition}
    For $\alpha \in (1,2), \ k\in \Z \setminus \{ 0 \}$ such that $2-1/|k| < \alpha < 2-1/|k|^2$ and any $T>0$, there exists a periodic $1$-precessing motion which revolves $|k|$ times around $0$, and period $T$.
    \label{prop:k-precessing-existence}
\end{proposition}
\begin{proof}
    For a given angular momentum $c$ and energy $h<0$, the angle between the apocenter and the pericenter is
    \begin{equation}
        \Phi(h,c) = \int_{\rho_{\min}}^{\rho_{\max}}\frac{|c|}{\rho\sqrt{2h\rho^2+2\rho^{2-\alpha}}-c^2}\d \rho,
    \end{equation}
    where $\rho_{\min},\rho_{\max}$ are the solutions of $V^c(\rho)=h$, where $V^c$ is the effective potential. It is known that $\Phi(h,c)\to \pi/(2-\alpha)$ for $c\to0^+$, and $\Phi(h,c)\to \pi/\sqrt{2-\alpha}$ for $h\to h_{\textrm{circ}}(c)$, where $h_{\textrm{circ}}(c)$ is the energy of the circular solution with angular momentum $c$. Therefore, by continuity of $\Phi$, there exist values of $(h,c)$ such that the precessing angle is exactly $\pi |k|$. 

    Because of the reflection symmetry of the central force problem with $\alpha$-homegeneous potential, the orbit is therefore closed, and the winding number is exactly $k$. By the time scaling symmetry, if $u(t)$ is a solution of the Euler-Lagrange equation with period $\tau$, then $u_\lambda(t) = \lambda^{2/(2+\alpha)}u(t/\lambda)$ is still a periodic solution with period $\lambda\tau$, therefore any given period $T$ can be matched.  
\end{proof}

\begin{remark}
    The existence of periodic $q$-precessing motions in the range of Eq.~\eqref{eq:alpha_q} is the same as above. Note that for $|k|=2,3,4$ there exist only $1$-precessing motions in the range $2-1/|k| < \alpha < 2-1/|k|^2$, while for $|k|\geq 5$ there exist also $q$-precessing motion for $q \geq 2$. 
    \label{rmk:q=1}
\end{remark}

\begin{proposition}
    Let $\alpha \in (1,2), \ k\in \Z\setminus\{0\}$ such that $\alpha < 2-1/|k|$. Then every $q$-precessing solution with winding number $k$ has action larger than the collision-ejection solution. 
    \label{prop:qprec}
\end{proposition}
\begin{proof}
    Without loss of generality, we can fix the energy to $h=-1$. Let us define $f(\rho) = 2\rho^{-\alpha} -2 -c^2/\rho^2$, with $c \in [0,c^*]$, with $c^*$ corresponding to the angular momentum of the circular motion. Then the reduced period $\sigma$, the precession angle $\Phi$, and the radial action $\psi$ can be written as
    \begin{equation}
        \sigma(c) = 2\int_{\rho_{\min}}^{\rho_{\max}} \frac{\d \rho}{\sqrt{f}}, \quad \Phi(c) = \int_{\rho_{\min}}^{\rho_{\max}} \frac{|c|}{\rho^2\sqrt{f}}\d\rho, \quad \psi(c) = 2\int_{\rho_{\min}}^{\rho_{\max}}\sqrt{f}\d\rho.
    \end{equation}
    Minimizing the action $\A^\alpha$ is equivalent to minimizing the reduced period $\sigma$. Note therefore that $\sigma(0)$ is related to the action of the collision-ejection solution, while the action of a $q$-precessing solution is related to $q\sigma(c_q)$. Note also that $q\Phi(c_q)=\pi |k|$.

    Set $\kappa = (2-\alpha)/(2\alpha)$, we have that $\sigma(c) = \kappa(\psi(c) - |c|\psi'(c))$, and $\psi'(c)=-2\Phi(c)$. Therefore
    \begin{equation}
        \sigma(c) = 2\kappa\bigg[ \int_{|c|}^{c^*}\Phi(s)\d s + |c| \Phi(c) \bigg].
        \label{eq:sigma_c_decomp}
    \end{equation} 
    In fact, set $F(h,c) = 2\int\sqrt{2h+2\rho^{-\alpha}-c^2/\rho^2}\d \rho$, we have $F_h = \sigma$ and $F_c= -2\Phi$.The substitution $\rho\to\lambda\rho$ gives $F(h,c)=|h|^{-\kappa}\psi(|h|^\kappa |c|)$. By substituting $h=-1$ we get Eq.~\eqref{eq:sigma_c_decomp}.

    Let us now prove that $\sigma(0) < q \sigma(c_q)$. From Eq.~\eqref{eq:sigma_c_decomp} we have:
    \begin{equation}
        q\sigma(c_q) - \sigma(0) = 2\kappa\bigg[ (q-1)\int_{c_q}^{c^*}\Phi(s)\d s + \bigg( \pi k c_q - \int_0^{c_q}\Phi(s)\d s \bigg) \bigg].
    \end{equation}
    The first term above is non-negative since $q\geq 1$. For the second term, $\Phi(s) < \pi/(2-\alpha)$, and by hypothesis this is smaller than or equal to $\pi |k|$ for all $s\in[0,c_q]$. Therefore the second term is positive, which finishes the proof. 
\end{proof}
\section{Generalization of Gordon's theorem }

\begin{proposition}
      For $\alpha \in (1,2)$ and $k\in \Z \setminus \{0\}$ such that $\alpha > 2-1/|k|$, the minimizer of $\A^{\alpha}$ is free of collisions.
    \label{prop:nocoll}  
\end{proposition}
\begin{proof}
    The proof is done by the blow-up technique \cite{venturelli-2002, ferrario-terracini_2004}. Let us fix $\varepsilon>0$ and define
    \begin{equation}
        \Omega_k^\varepsilon = \{ H^1(\R, \R^2 \setminus \{0\}): u(0) = u(T), \ \deg(u, 0)=k, \ \norm{u(t)}>\varepsilon \ \forall \ t\in[0,T] \}.
    \end{equation}
    Finding the minimum of $\A^\alpha$ on $\Omega_k^\varepsilon$ is therefore an obstacle problem, and $\min_{\overline{\Omega}_k^\varepsilon} \A^\alpha \geq \min_{\overline{\Omega}_k}\A^\alpha$. Let us denote with $u^*$ the minimizer on $\overline{\Omega}_k$ and with $u_\varepsilon$ the minimizer on $\overline{\Omega}_k^\varepsilon$. Let us suppose that $u^*$ has a collision, then $\A^\alpha(u_\varepsilon)$ converges to $\A^\alpha(u^*)$, and $u_\varepsilon$ necessarily touches the border of the disk of radius $\varepsilon$. In fact, if it did not touch the border, then $u_\varepsilon$ would be a smooth solution of the Euler-Lagrange equation, and it would therefore be a minimum of the original problem.

    By the blow-up technique, the sequence of $u_\varepsilon$ converges to $u_0$, where $u_0$ is a parabolic collision-ejection solution. Note that the precessing angle swept by $u_0$ is the limiting value $\pi/(2-\alpha)$. For the obstacle solution $u_\varepsilon$, denote $\theta_\varepsilon^{(i)}$ the angle swept near the $i$-th touch of the $\varepsilon$-disk, then
    \begin{equation}
        \liminf_{\varepsilon\to0}\theta_\varepsilon^{(i)} \geq \frac{2\pi}{2-\alpha}, \quad i=1,\dots,m,
    \end{equation}
    where $m$ is the number of collisions per period in the colliding minimizer. Since $u_\varepsilon \in \overline{\Omega}_k^{\varepsilon}$, then necessarily $\sum_{i=1}^m \theta_\varepsilon^{(i)} \leq 2k\pi$. Taking the $\liminf$ we obtain 
    \[
    m \frac{2\pi}{2-\alpha} \leq \liminf_{\varepsilon\to0} \sum_{i=1}^m \theta_\varepsilon^{(i)} \leq 2|k|\pi,
    \]
    which leads to $m\leq k(2-\alpha)$. Since the collision requires $m\geq 1$, then $|k|(2-\alpha) \geq 1$ and therefore $\alpha \leq 2-1/|k|$. On the other hand, if $\alpha > 2-1/|k|$, then the minimizer is necessarily free of collisions. 
\end{proof}

\begin{theorem}
    For $\alpha \in (1,2)$ and $k\in \Z \setminus \{0\}$, the minimizer of $\A^{\alpha}$ on $\overline{\Omega}_k$ is:
    \begin{itemize}
        \item[i)] the ejection-collision solution for $\alpha \in (1, 2-1/|k|]$;
        \item[ii)] a periodic $q$-precessing solution for $\alpha \in (2-1/|k|, 2-1/|k|^2)$;
        \item[iii)] the $k$-circular solution for $\alpha \in [2-1/|k|^2, 2)$. 
    \end{itemize}
    \label{thm:gordon_general}
\end{theorem}
\begin{proof}
    \emph{i)} Because of Corollary~\ref{cor:circ_not_min}, the $k$-circular motion is not a minimizer of the action. From Proposition~\ref{prop:k-precessing-existence}, $1$-precessing solutions do not exist in this interval of $\alpha$ values. Because of Proposition~\ref{prop:qprec}, the $q$-precessing solutions with $q \geq 2$ are not minimizer either, since their action is larger than the action of the collision-ejection solution. Therefore there are no other competitors which are smooth solution of the Euler-Lagrange equation, hence the only possibility left is that the minimizer is the ejection-collision solution. 
    
    \emph{ii)} From Proposition~\ref{prop:nocoll}, the minimizer is collision free, therefore it is a smooth solution of the Euler-Lagrange equation. From Corollary~\ref{cor:circ_not_min}, the $k$-circular solution is not a minimizer. Therefore, the only possibility is that the minimizer is a periodic $q$-precessing solution. 
    
    \emph{iii)} From Proposition~\ref{prop:nocoll}, the minimizer is collision free, therefore it is a smooth solution of the Euler-Lagrange equation. Because of Proposition~\ref{prop:k-precessing-existence}, periodic $q$-precessing motions do not exist for $\alpha > 2-1/|k|^2$, therefore the only possibility is that the $k$-circular solution is the minimizer.
\end{proof}

\begin{remark}
    By Remark~\ref{rmk:q=1}, for $|k|=2,3,4$ the $1$-precessing periodic motion is the only minimizer in case \emph{ii)} of Theorem~\ref{thm:gordon_general}. 
    For $|k|\geq 5$, $q$-precessing periodic motions for $q\geq2$ also exist, and therefore compete among the possible minimizers.
    \label{rmk:open}
\end{remark}

\section{Numerical example for $k=2$}
\label{s:numeric}
We verified numerically Theorem~\ref{thm:gordon_general} for $k=2$. The action of the $2$-circular solution is expressed by an analytical formula, while the action of the ejection-collision solution is computed by numerical quadrature (see \cite{venturelli-2002}). On the other hand, the 1-precessing solution for $\alpha \in (3/2, 7/4)$ has to be found numerically. We therefore followed the approach presented in e.g. \cite{fenucci-gronchi_2018, fenucci-gronchi_2023}. We first represented periodic loops with truncated Fourier series, and applied a gradient descent method \cite{simo_2000, simo_2001} to get close to the minimizer for $\alpha = 1.7$. This is value is chosen so that the 1-precessing solution is far enough from the collision, for which the gradient descent does not perform well, and far enough from the circular orbit as well. The minimizing loop is then refined using a shooting method to find the precise initial condition of the 1-precessing orbit with period $T=2\pi$. Both the gradient descent and the shooting method are implemented in polar coordinates, so that the winding number can be easily fixed. Once the solution for $\alpha=1.7$ was found, we used a continuation method (see e.g. \cite{fenucci-jorba_2020}) to find the 1-precessing solution on the whole interval $\alpha \in (3/2, 7/4)$. The action is also computed by numerical quadrature. Figure~\ref{fig:minaction} shows the action of the three different type of solution. Note that the continuation method failed when $\alpha$ approached the extrema values of 3/2 and 7/4, indicating that the 1-precessing solution does not exist outside of the interval. The figure shows that the 1-precessing solution is the minimum over the whole interval $(3/2, 7/4)$, the collision-ejection solution is the minimum for $\alpha\leq3/2$, and the 2-circular solution is the minimum for $\alpha \geq7/4$, as stated in Theorem~\ref{thm:gordon_general}. 
\begin{figure}
    \centering
    \includegraphics[width=0.9\linewidth]{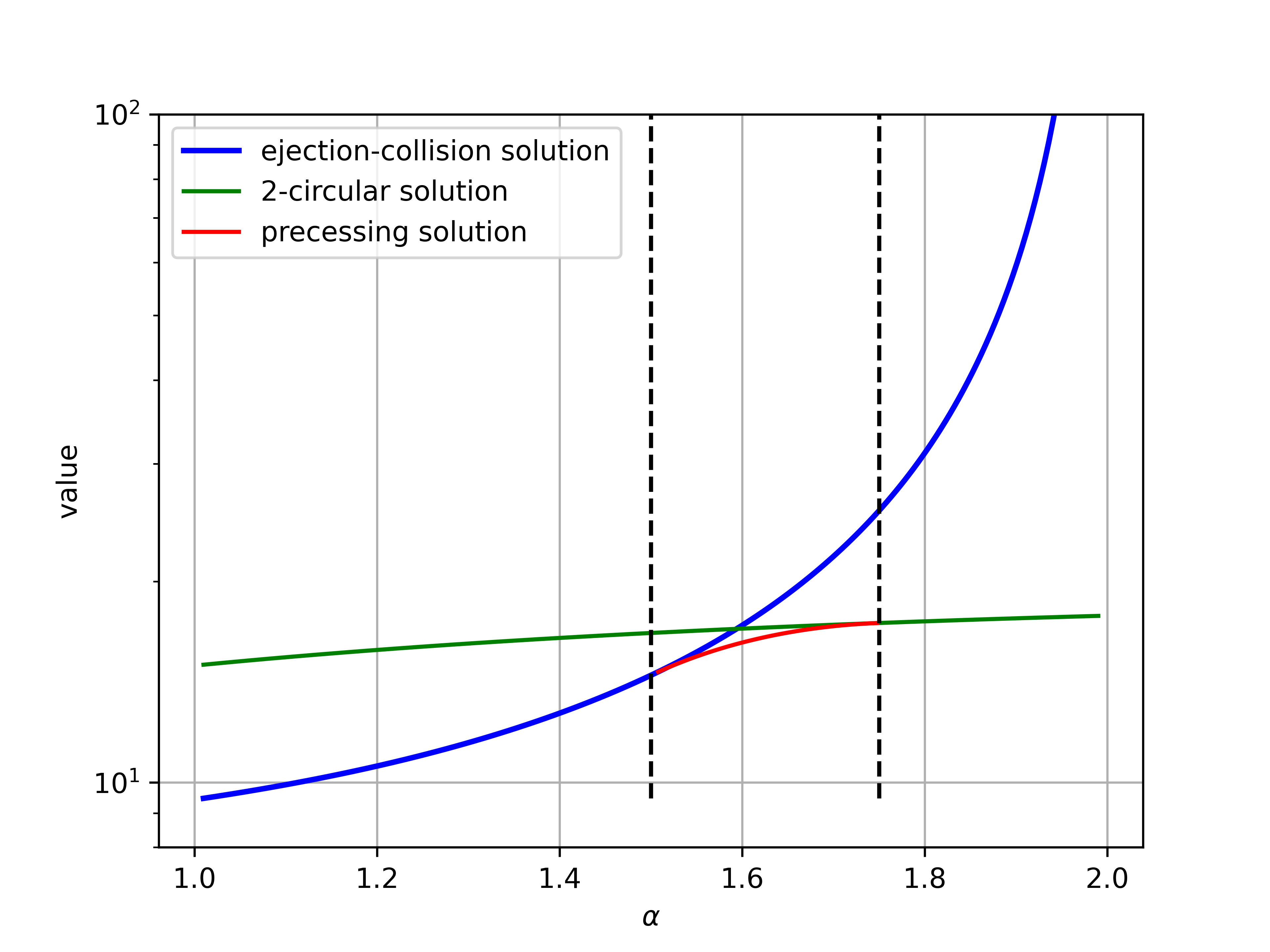}
    \caption{Action of the three competitors for the minimum, for $k=2$ and $T=2\pi$. In blue, the action of the ejection-collision solution. In green, the action of the $2$-circular solution. In red, the action of the precessing orbit. The vertical dashed lines correspond to $\alpha=3/2$ and $\alpha=7/4$.}
    \label{fig:minaction}
\end{figure}

Figure~\ref{fig:solutions} shows the solutions obtained by the shooting method for some values of $\alpha \in (3/2, 7/4)$. Note that for $\alpha$ approaching $3/2$, the 1-precessing solution gets stretched in one direction, and the precessing at pericenter is not even noticeable in a graphical representation because it gets very close to the singularity. This indicates that the solution tends to the ejection-collision solution. As $\alpha$ increases, the minimum radius becomes larger and, as $\alpha$ approaches the value of 7/4 it tends to the $2$-circular solution.
\begin{figure}
    \centering
    \includegraphics[width=0.9\linewidth]{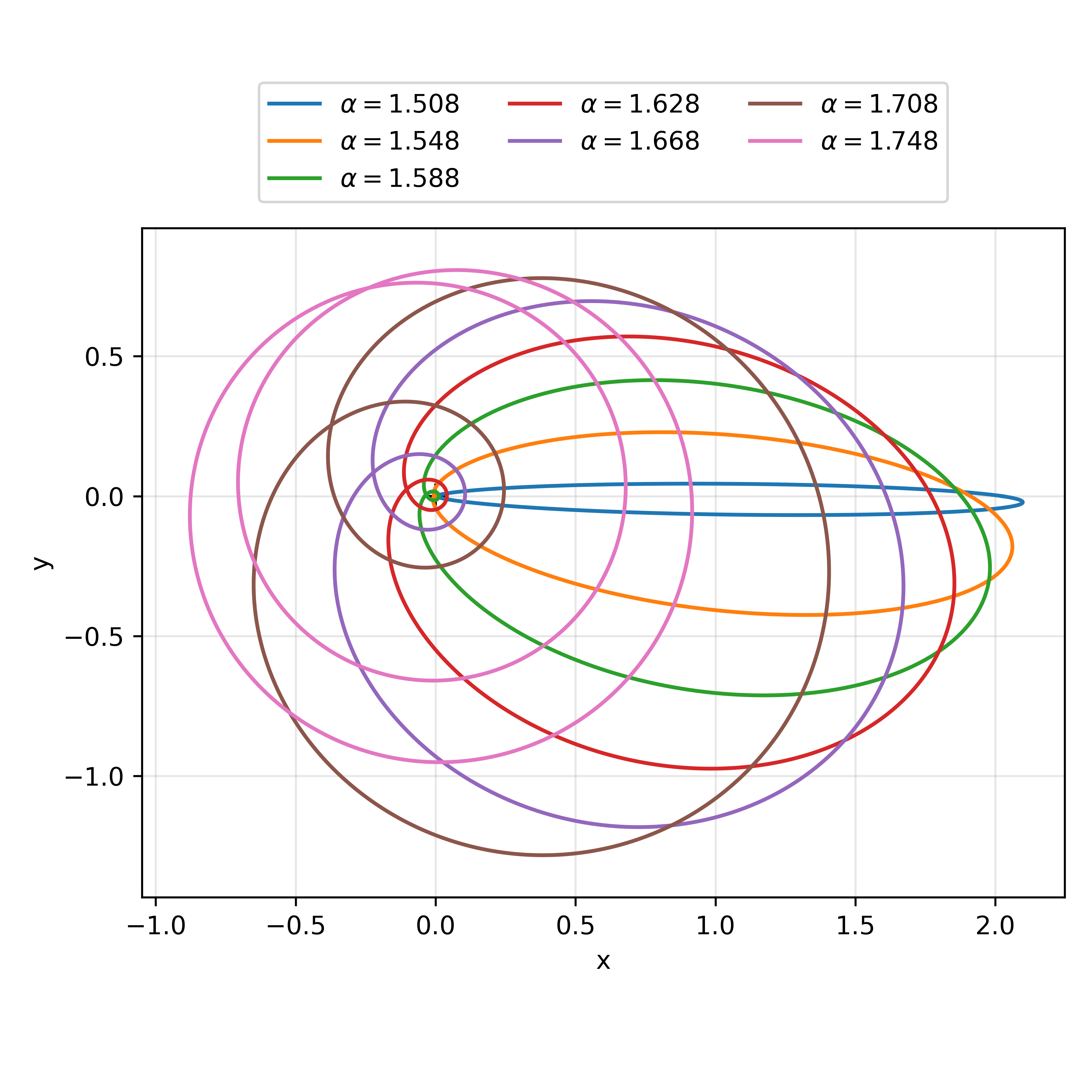}
    \caption{Precessing solutions for $k=2$ with period $T=2\pi$, obtained by the shooting method and numerical continuation.}
    \label{fig:solutions}
\end{figure}

\section{Conclusions}

We have completed the classification of minimizers of the Keplerian action functional with homogeneous potential of type $1/|u|^\alpha$, for $\alpha \in (1,2)$ and arbitrary winding number $k \in \Z \setminus \{0\}$, thereby closing the case left open by Gordon and Venturelli. Depending on the position of $\alpha$ relative to the two thresholds $2-1/|k|$ and $2-1/|k|^2$, the minimizer is either the collision-ejection solution, a periodic precessing solution, or the $k$-circular solution.

Some questions remain open. For $|k| \geq 5$, Remark~\ref{rmk:open} shows that several families of $q$-precessing solutions with $q \geq 2$ coexist in part of the intermediate range. Our results do not identify which of these competing families is the actual minimizer, nor rule out the possibility that the minimizer changes from one family to another as $\alpha$ varies within the range, although it is reasonable to conjecture that the $1$-precessing orbits are always the minimizers. Resolving this would complete the classification at the level of individual solutions rather than just the three broad classes considered here. It would also be of interest to extend the numerical study of Section~\ref{s:numeric} beyond $k=2$, to check whether the qualitative picture of Figure~\ref{fig:minaction} persists once several precessing families compete for the minimum. 

\section*{Declarations}
\noindent\textbf{Funding.} The author declares that no funds, grants, or other support were received during the preparation of this manuscript.

\noindent\textbf{Competing interests.} The author has no relevant financial or non-financial interests to disclose.

\noindent\textbf{Author contributions.} M. F. is solely responsible for the conceptualization, analysis, and writing of this manuscript.

\noindent\textbf{Data availability.} Data sharing is not applicable to this article, as no datasets were generated or analysed during the current study.

\bibliographystyle{amsplain}
\bibliography{sample}

\end{document}